\documentclass{article}
\usepackage{makeidx}
\usepackage{amssymb , amsmath, amsthm}
\usepackage{amsmath}
\usepackage{amssymb,latexsym}
\usepackage[dvips]{graphicx}
\usepackage{color}
\usepackage{variations}
\usepackage[latin1]{inputenc}
\usepackage{soul}

\newcommand{\gc}{\tilde g}

\newcommand{\nero}{\smallskip$\bullet\quad$}

\newcommand{\h}{\mathbb H}

\newtheorem{theorem}{Theorem} [section]
\newtheorem{lemma}{Lemma} [section]
\newtheorem{proposition}{Proposition} [section]
\newtheorem{corollary}{Corollary} [section]

\newtheorem{definition}{Definition} [section]

\newtheorem{remark}{Remark}[section]

\newcommand{\mnote}[1]
{\protect{\stepcounter{mnotecount}}$^{\mbox{\footnotesize
$
\bullet$\themnotecount}}$ \marginpar{
\raggedright\tiny\em $\!\!\!\!\!\!\,\bullet$\themnotecount: #1} }
\newcounter{mnotecount}[section]

\def\lra{\longrightarrow}
\def\Rn{\mathbb{R}^{n+1}}

\def \Rn {\mathbb{R}^{n+1}}

\begin{document}
\medskip
  \title{ On the Entropies of Hypersurfaces with bounded mean curvature}

\author{  S. Ilias, B.Nelli, M.Soret }

\date{}

\maketitle

\begin{abstract} 
We are interested in the impact of entropies on the geometry of a hypersurface of a Riemannian manifold.  In fact, we will be able to compare the volume entropy of a  hypersurface with that of the ambient manifold, provided some geometric assumption are satisfied. This depends on the existence of an embedded  tube  around such  hypersurface.  Among the  consequences of 
 our study of the entropies, we point out some new answers  to a question of do Carmo on stable Euclidean hypersurfaces of constant mean curvature.\\
\end{abstract}
 

\section{Introduction}

 For a complete noncompact Riemannian manifold, there are many results relating the exponential volume growth of geodesic balls to some of its Riemannian invariants. 
 An important result  concerning this asymptotic invariant is that obtained by Brooks, which gives an upperbound of the bottom of the essential spectrum in terms of the exponential volume growth of geodesic balls (see \cite{Br} and Theorem \ref{brooks}). 
 As we will  introduce other growth functionals, we will call such exponential volume growth, {\it volume entropy}, even if this is commonly used for the exponential volume growth of the universal cover of a compact manifold.

 The main purpose of the present paper is to give new results concerning various entropies for hypersurfaces of spaces of bounded geometry. A first immediate  remark is an extrinsic generalization of the result of Brooks. In fact, we will introduce an extrinsic entropy for hypersurfaces and show that  Brooks' upperbound of the bottom of its essential spectrum  is still valid with this new  extrinsic entropy (see Corollary 
 \ref{ext-brooks-theo}).  
 
Then, we   prove a general embedded tube result concerning hypersurfaces with mean curvature bounded away from zero 
(see  Theorem \ref{tubeplonge}).  More precisely, we prove the existence of an embedded tube 
of fixed radius around any hypersurface with bounded curvature and mean curvature bounded away from zero, properly embedded in a simply connected manifold with bounded geometry  (see also \cite{MT}).  
  The embeddedness of this tube, apart of its own interest, allows us to compare the volume of extrinsic balls of such hypersurfaces with that of the geodesic balls of the ambient space (Corollary \ref{volume-estimate}). 
This  gives, in particular, a comparison between the extrinsic volume entropy of the hypersurface and the (intrinsic) 
volume entropy of the ambient space. 
 In order to relate  the volume entropy to the total curvature entropy (see Section \ref{entropy-section} for the Definition) of  finite index hypersurfaces with constant mean curvature we use some Cacciopoli's inequalities 
(obtained in \cite{INS}).
Finally, using  these results about entropies,  we obtain 
some relations between spectra and curvature and various nonexistence results concerning some classes of hypersurfaces, including those of constant mean curvature stable or of  finite index. Among the problems considered, we will be interested in a question of do Carmo about Euclidean stable hypersurfaces of constant mean curvature. In fact, M. do Carmo (\cite{DoC1}) asked the following : "Is a noncompact, complete, stable, constant mean curvature  hypersurface of $\Rn,$ $n\geq 3,$  necessarily  minimal? "
The answer was already known to be positive for $n=2$ \cite{LR}, \cite{S}, \cite{Pa}. Later,  the answer was proved to be positive for $n=3,4,$  by  M.F. Elbert,  the second author and H. Rosenberg  \cite{ENR} and independently by  
X. Cheng \cite{X},  using  a Bonnet-Myers's type method. So far, stability alone - or more generally finite index - does not seem to yield  the answer to do Carmo's question in higher dimensions. Note that  do Carmo's question can also be asked for general ambient spaces. We observe also, that  R. Schoen, L. Simon and S.T. Yau \cite{SSY} proved that a  properly immersed orientable stable minimal hypersurface of ${\mathbb R}^{n+1}, $ 
$n \leq 5,$ which has Euclidean volume growth  - hence zero volume entropy - is a hyperplane.\\
 Using our results about entropies, we give a positive answer to do Carmo's question if the volume entropy of $M$ is zero  (Corollary \ref{carmo-euclidean-hyp}) ( the  ambient manifold  being arbitrary) or if the total curvature  entropy  of $M$ is zero and $n\leq 5$ (Theorem \ref{general-theo})  ( the ambient manifold being a space-form). 
In particular, the comparison between the volume entropy of the ambient space and that of the hypersurface will also gives a positive answer if the hypersurface  $M$ has bounded curvature and is properly embedded (Theorem \ref{embedded-hyp-eucl})
 ( the  ambient manifold  being a simply-connected manifold with bounded geometry and with zero volume entropy). 
In the same spirit, we obtain nonexistence results for finite index hypersurfaces.\\ 

 The plan of the paper is as follows: 
In Section \ref{stability}, we recall the notion of stability, finite indices, and establish a useful  analytic Lemma (Lemma \ref{max-prin}), depending on the maximum principle. The construction of embedded half-tubes around  a complete  hypersurface with bounded curvature and mean curvature bounded away from zero, that is  properly embedded in a simply-connected manifold of bounded curvature,  is carried out in Section \ref{tubes}.  Some results from \cite{INS} about Caccioppoli's inequalities for constant mean curvature hypersurfaces of finite index are recalled in  Section \ref{caccioppoli}.
Next, in Section \ref{entropy-section},   various notions of entropies are introduced and discussed.
In Section \ref{caccio-appli},  we apply Caccioppoli's inequalities in order to relate the volume entropy to the total curvature entropy. Finally, in Section \ref{docarmo-section} we  answer positively to do Carmo's question in the cases described above.  
Some of the arguments used in the proof of the halftube Theorem \ref{tubeplonge} and which are of independent interest  are detailed in Section 8 (Appendix).

{\bf Acknowledgements.} The second author would like to thank the LMPT of Universit\'e Fran\c{c}ois Rabelais de Tours, for hospitality during the preparation of this article.

    
  \section{Stability of an operator $L$ and of a manifold $M$ } 
\label{stability}

Consider a Riemannian manifold $M$ and the operator $L=\Delta+V$ where $\Delta=tr\circ Hess$ on $M$ and $V$ is a smooth potential. 
Associated to $L$ one has the quadratic form

\begin{equation*}
Q(f,f) := -\int_M f Lf
\end{equation*}
 defined on $f\in C_0^{\infty}(M).$ 
 
 Let $\Omega$ be a relatively compact domain of $M.$ Define  $i_{L|\Omega}$ (respectively $Wi_{L|\Omega}$) the number of negative eigenvalues of the operator  $-L,$ for the Dirichlet problem on $\Omega:$ 

$$-Lf=\lambda f,\ \   f_{|\partial\Omega}=0$$
$$  ({\rm resp.}  \  -Lf=\lambda f,\   f_{|\partial\Omega}=0,\  \int_{\Omega}f=0).$$ 

The ${\rm Index}(L)$ (resp. ${\rm WIndex}(L)$) is defined as follows

\begin{align*}
&{\rm Index}(L):=\sup\{ i_{L|\Omega} \ | \ \Omega\subset M \ {\rm rel.\ comp.}\} \\ 
&({\rm resp.} \ {\rm WIndex}(L):=\sup\{  {\sl Wi}_{L|\Omega} \ | \ \Omega\subset M \ {\rm rel.\ comp.}\})
\end{align*}

The operator $L$ is said {\em nonpositive}   (respectively {\em weakly nonpositive}) if  ${\rm Index}(L)=0$ (respectively  ${\rm WIndex}(L)=0$).

We will look at the following  situation in which the operator $L$ has a geometric interpretation.

Assume that $M$ is a constant mean curvature hypersurface in a manifold $\mathcal N.$
 It is known  that constant mean curvature hypersurfaces  are critical  for  the area functional  with respect to  compact support deformations that keep  the  boundary  fixed and  whose algebraic volume swept  in $\mathcal N$ during deformation remains zero  (see for instance    \cite{BdC} and \cite{BdC2}).
    The operator   $L:=\Delta+Ric(\nu,\nu)+|A|^2$ is called the {\em stability operator} of $M$. Here $Ric$, $\nu$ and $A$ are respectively the Ricci curvature, the unit normal field and the second fundamental form of the hypersurface $M$.
We define the ${\rm Index}(M):={\rm Index}(L)$ and ${\rm WIndex}(M):={\rm WIndex}(L).$ Moreover $M$ is said to be {\em stable} (respectively {\em weakly stable})  if $L$ is nonpositive (respectively {\em weakly nonpositive}).

Stability (respectively weak stability) of $M$  means that 

\begin{equation*}
Q(f,f)\geq0,\ \ \forall f\in C_0^{\infty}(M)\ \ ({\rm resp.}\ \forall f\in C_0^{\infty}(M)\ \int_Mf=0)
\end{equation*}

It is easy to see that $Q(f,f)$ is the second derivative of the volume in the direction of $f\nu$ (see \cite{BdC}),  then ${\rm Index}(M)$  (respectively ${\rm WIndex}(M)$)  measures the number of linearly independent normal deformations with compact support of $M,$  that decrease  area
(respectively  that  decrease    area, leaving fixed a volume).
When $H=0,$ one can drop the condition $\int_{M}f=0,$ then, for a minimal  hypersurface one consider only the ${\rm Index}(M).$ 
It is proved in \cite{BB} that ${\rm Index}(M)$  is finite if and only if  ${\rm WIndex}(M)$  is finite. So, when we assume 
finite index we  are referring to either of the indexes   without distinction. 
Finally, we recall that   $M$ has finite index if and only if there exists a compact subset $K$ of $M$ such that $M\setminus K$ is stable  (see Proposition 2.1 of \cite{INS}).
Later on, we will need an analytic Lemma that summarizes some properties of the  operator $L=\Delta+V,$ where  $V\in L^1_{loc}(M)$ is a 
suitable potential. 

\begin{lemma} \label{max-prin} 
  Let $M$ be a complete manifold and $L= \Delta + V$, 
  $V\in L^1_{loc}(M).$ 
Consider the following assertions:

(1)  For any compact 
 domain $\Omega \subset M,$ $L$ satisfies the maximum principle: for  $u,v \in W^{1,2}(\Omega)$ such that
    $u\leq v $ on $\partial \Omega$ and 
$Lu \geq Lv, $  one has $u \leq v$ on $\Omega$.

(2) The Dirichlet problem for $L$ on any  compact domain $\Omega \subset M$ has a unique solution in $W^{1,2}(\Omega)$.

(3) The operator  $L$ is nonpositive   on $M$. 

(4) There exists a  smooth positive $\phi$  on $M$ such that $L \phi \leq 0.$

(5) There exists a smooth  positive $\phi$  on $M$ such that $L \phi = 0.$
\

Then  (3),(4),(5) are equivalent and any of the latter assertions implies (1) which implies (2).
\end{lemma}
 \begin{proof}

For the equivalence between  (3), (4), (5) see   \cite{FC}, \cite{FCS} and \cite{MPR} (if $M$ is an Euclidean space, such equivalence was originally due to   Glazman  (unpublished) and also  to Moss \& Pieperbrinck \cite{MP}).\\
  (1) $\implies$ (2): apply (1) to $u-v$ and $v-u$ and use linearity of $L$.\\
 (5) $\implies$ (1): by  the linearity of $L,$ it is enough to prove that if $w\in W^{1,2}(\Omega)$ such that $w\leq 0$ on $\partial \Omega$ and 
 $Lw\geq 0$ on $\Omega,$ then $w\leq 0$ on $\Omega.$ Define $h:=\frac{w}{\phi},$ where $\phi$ is as in (5). By a straightforward computation one has that 

 $$\Delta h + 2\langle \nabla h, \frac{\nabla\phi}{\phi}\rangle \geq 0.$$

By the maximum principle (see for instance  Theorem 10.1 \cite{GT}),  one has  $h\leq 0$ i.e. $w\leq 0$ on $\Omega.$
\end{proof}

 We need the following application of the maximum principle ((1) of Lemma  \ref{max-prin}).
 
\begin{corollary} \label{basic2}
Suppose $L$ is nonpositive and  there exists  a nonnegative $v\in W^{1,2}(M)$  such that  $Lv\leq -c$ for some positive constant $c.$  If there exists   a 
positive $u\in W^{1,2}_0(M)  $   
  such that $L u> -c$, then  $v\geq  u$ on $M$.
\end{corollary}

\begin{proof}
Let $\Omega$ be the support of $u$, then $v\geq u$ on $\partial \Omega$. By hypothesis $Lv\leq -c\leq L u$.
 Since $L$ is nonpositive,   the  maximum principle ((1) of Lemma  \ref{max-prin})   yields $v\geq u $ on $\Omega$ and hence 
 on $M$.
\end{proof}
 
\begin{remark}  
 Functions similar to functions $u, v$  of   Corollary \ref{basic2} are studied in \cite{NK}.
\end{remark}

  %
  %
   
\section{Tubes around embedded   hypersurfaces of  mean curvature bounded from below}   
\label{tubes}

Throughout this section,  the manifold ${\mathcal N}$ is  a simply-connected manifold with bounded curvature.
	Let $M$  be an orientable  hypersurface  {\it properly embedded}  in ${\mathcal N}$ and assume that $M$ 
has  mean curvature   function  $H$  bounded away from zero. Orient $M$ by its mean curvature vector $\vec H.$ Then $\vec H=H\vec \nu$,
 and $H$ is positive.    
Moreover  we assume that $M$ has bounded second fundamental form.
Since $\mathcal N$ is simply connected, by   Jordan-Brouwer Separation  Theorem,
  $M$ separates $\mathcal N$ into two components (see  for instance \cite{Lim}).
We call  {\it mean convex} side of $M$ the component towards 
  which $\vec H$ points.
  We consider  the normal exponential map of $M$ defined by 
\begin{equation}   
\begin{array}{lll}
  exp:& M\times \mathbb{R}&\longrightarrow \mathcal N\\
& (p,r) &\mapsto exp_p(r\nu)
 \end{array}
\end{equation}

As $M$ and $\mathcal N$ have bounded curvature,  for any $p\in M,$ there exists a geodesic ball  $B_{r_0}(p)$  centered at $p$ 
   of radius  $r_0>0$ such that   $exp_{|B_{r_0}(p)\times(-r_0,r_0)}$ is a diffeomorphism on its image. 
 
Let 
   $$T^+( r_0) :=     exp\left( M \times (0,r_0)\right)$$
   be   the  {\it   half-tube } in $\mathcal N$  around  $M$ of radius $r_0.$ The half-tube   $T^+( r_0)$ is  locally diffeomorphic to $M \times (0,r_0).$
In this section, we prove the following  embeddness  property  of  the half-tube  $T^+( r_0)$  which  will be crucial in the proof of one of  the main results of this article.

 \begin{theorem}\label{tubeplonge}
      Let $M  $ be a  complete  hypersurface  properly embedded in a simply connected manifold $\mathcal N$.  Assume 
      that  $M$ and $\mathcal N$  have    bounded curvature  and $M$ has possibly  non empty compact boundary.   Assume that $M$ has  mean curvature  function  $H$  bounded  away from zero ($H\geq \frac{\epsilon}{n} >0$ ). Then one has the following results.
     
     (1) If $\partial M =\emptyset$, there exists an embedded half-tube 
      $T^+( \rho)$   contained in the mean-convex  side  of $M,$ where the radius $\rho$ depends  on the curvatures of $M$ and 
      $\mathcal{N}$.\\ 
     (2)  If $\partial M $ is compact,  then there exists a compact subset $K$ of $M$ and a half-tube 
      $T^+( \rho') $  of $M\setminus K$   contained in the mean-convex  side  of $M,$ where the radius $\rho'$ depends  on the   compact $K,$ and the  curvatures of $M$ and   $\mathcal{N}$.
      \end{theorem}

   \begin{remark}   
\nero In the case where $M$ has no boundary and constant mean curvature, the existence of a tube of constant radius around $M$ was stated in \cite{MT} (see also  \cite{So} and \cite{NS}). We believe that it is necessary and useful to give a detailed proof of this result, especially in the more general setting that we are treating.\\
\nero  If $\mathcal{N}$ is not simply-connected, 
    assuming  in addition that $M$ is  strongly Alexandroff embedded (see Definition 3.3 in \cite{MT}), one obtains the same conclusion as in Theorem \ref{tubeplonge}.     
   \end{remark}        
    
\begin{center} \begin{figure}[h]
\hskip 2 in  \includegraphics[scale =  .2]{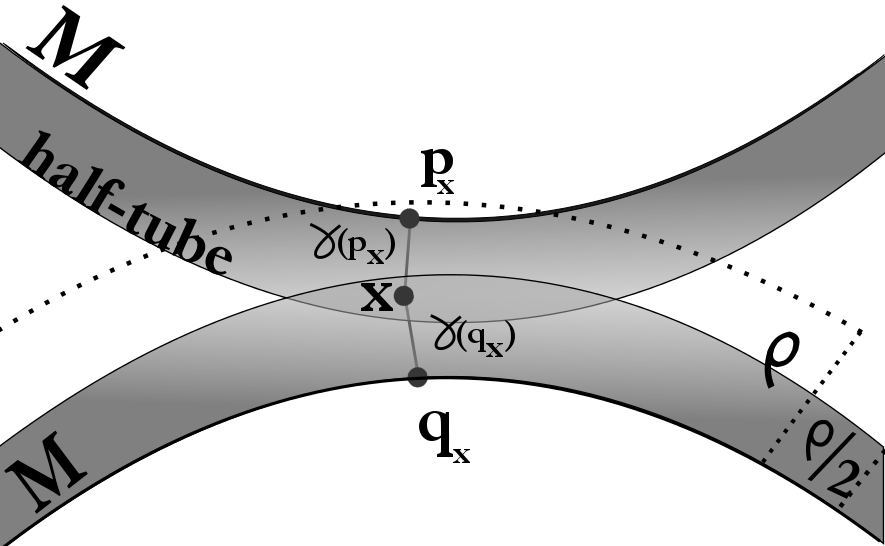}
\caption[]{  Self-intersection of a halftube of $M$ in  $\mathcal N$}
\end{figure}
 \end{center}

 \begin{proof}[ Proof of Theorem \ref{tubeplonge}]
 (1)  See Figure 1.
 Let $r_1>0 $  be such that for any $p\in M,$  $exp_{|B_{\rho}(p)\times(0,r_1)}$ is an isometry
  on  $B_{p}(\rho)\times(0,r_1)$  endowed with the pull-back metric.  The proof is in two steps.
In the first step, we will prove that if there is a  positive $\rho\leq r_1 $ such that    $T^+(\rho) \cap M= \emptyset$,  then 
   $T^+(\frac{\rho}{2})$ is embedded. In the second step we prove that, 
      for $\rho$ sufficiently small, we have  $T^+(\rho) \cap M = \emptyset.$\\
   \parindent=10pt 
  {\it Proof of Step 1.}   
 \parindent=0pt 
  Assume, by contradiction that $T^+(\frac{\rho}{2})$ is not embedded.  Then  there exist   $x\in T^+(\frac{\rho}{2})$   and $p_x\not=q_x\in M,$ such that  $x$ belongs to the geodesic $\gamma(p_x)$ starting at $p_x,$ orthogonal to $M$ and also belongs to the geodesic $\gamma(q_x)$ starting at $q_x,$ orthogonal to $M.$ Moreover $d(p_x,x)\leq \frac{\rho}{2}$ and 
   $d(q_x,x)\leq \frac{\rho}{2}$ where $d$ is the distance in $\mathcal N.$  By the  distance  inequality,  $d(p_x,q_x)\leq  \rho.$ Then $q_x$ is contained in a  geodesic ball  in $\mathcal N,$ centered at $p_x,$ of radius $\rho.$   As $\rho\leq r_1,$ there are $(p,t),$ $(q,s),$ $(x_0,l),$ contained in a geodesic ball   $B_{(p,t)}(\rho)$ of  $M\times(0,\rho]$ centered at $(p,t)$ of radius $\rho$   and such that 
   $exp(p,t)=p_x,$   $exp(q,s)=q_x,$  $exp(x_0,l)=x.$ Without loss of generality, we  can assume that 
   $t=0.$  If $s=0,$ then  both  geodesics $exp^{-1}(\gamma(p_x)),$   
   $exp^{-1}(\gamma(q_x)),$ would be orthogonal to $M\times\{0\}$ and would meet at    $(x_0,l)\in B_{(p,0)}(\rho),$ which  is contradiction  since  $exp$ is an isometry on  $B_{(p,0)}(\rho).$ Then $s\not=0$ and $q_x\not\in M.$ Contradiction.     Therefore $T^+(\frac{\rho}{2})$ is embedded. 

 \

\begin{center}   \begin{figure}[h]   %
\mbox{\hskip 1.7 in 
 \includegraphics[scale = .35]{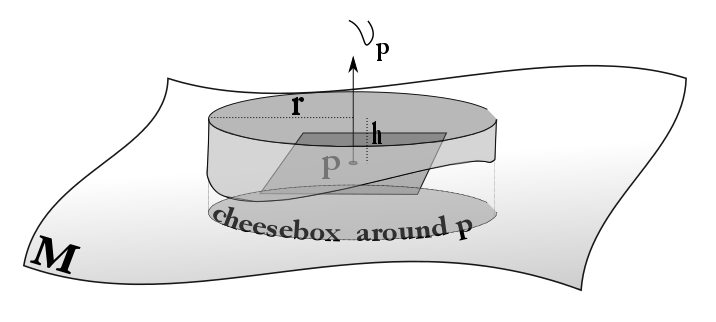}
 }
\caption[]{A cheesebox around $p\in M \subset \mathcal{N}$}
\end{figure}
\end{center}

\parindent=10pt 
{\it Proof of Step 2.}
 \parindent=0pt      
Assume by contradiction that  for any sufficiently small $\rho>0$, $T^+(\rho) \cap M\not= \emptyset.$  Let $\tilde p$ be a point in   $T^+(\rho) \cap M$ and let   
$(p,t)\in M\times(0,\rho]$ be the nearest point to $M\times \{0\}$  in $exp^{-1}(\tilde p)$ which is 
 on the geodesic starting at $(p,0)\in M\times\{0\}$  and orthogonal to $M.$  Define  a function $\phi$  by $\phi(p)=t$ and let $\Omega$ be the set of points $p$ of $ M$ such that $t\leq\rho.$    Since  $M$ is properly embedded, $\phi(p) >0$.

Let   $P:=\{(p,\phi(p)) \ | \ p\in\Omega\}.$   The fact that    $M$ is  embedded implies that $P\cap \left(M\times\{0\}\right)= \emptyset$.

We claim that  the function $\phi$ is bounded uniformly in $C^1(\Omega)$ and $||\phi||_{C^1(\Omega)}$ goes to zero as $\rho\longrightarrow 0.$ We prove these facts  by using the so called {\it cheesebox argument} (see  the details in Appendix \ref{cheese}, see also \cite{So}).\\

 We recall   the  cheesebox argument: 
Let $M$ be a complete  hypersurface with  bounded curvature  in  a manifold  $\mathcal N$ with  bounded curvature. Then, there exists   a   positive $\rho$  such that, for any point $p\in M $ (see Figure 2):  \\
 
 (a) The hypersurface $M$ is locally a graph $G_p$ of a function $\phi$ defined on a ball $D(p,\rho)$ of  $T_pM,$ centered at $p$ of radius 
 $\rho.$\\
(b) The graph $G_p$ is contained in  the cheesebox $D(p,\rho) \times [-h/2,h/2],$ of radius $\rho$ and height $ h = c\rho^2$  (where $c$ is a constant   depending    on   bounds of the curvatures of $M$ and $\mathcal N$). Moreover   $G_p$ cuts the boundary of the cheesebox  only on $ \partial D(p,\rho)\times [-c\rho^2,c\rho^2].$\\
 It is proved in  Appendix \ref{cheese} that properties (a) and (b) imply the following  estimate (see inequality \eqref{phi}):
 \begin{equation}\label{phi-estimate}
  ||\phi||_1 := \left(\sup_{p\in \Omega  } 
  | \phi|  + \sup_{p\in \Omega  } 
  |\nabla\phi| \right)   \leq O(\rho).
 \end{equation} \\
One concludes in particular  that  the tangent plane to $M\times\{0\}$ at $p$ and the tangent plane to $P$ at $p'$ are  as close as one wishes if $\phi$ is small.
From Appendix \ref{cmcequation}, $\phi$ satisfies a uniformly elliptic quasilinear 
partial differential equation (see equation \eqref{cmc30}). By a  classical    argument in elliptic theory, the derivatives of any order of $\phi $ are uniformly bounded (see for instance \cite{GT} Theorem 6.2 and  Problem 6.1).
Thus $||\phi||_{C^\infty(\Omega)}$ is uniformly bounded. \\

\

 We claim that the mean curvature vector  of $P$ points towards $M\times \{0\}.$
  Indeed,  let $(p,0)\in \Omega\times \{0\}$ and $(p,\phi(p))$ the corresponding point on $P.$  Let $\gamma$ be the geodesic, orthogonal 
to $M\times\{0\}$ joining $(p,0)$ to $(p,\phi(p)).$ The geodesic  
$\gamma$  in $\mathcal N$ does not intersect $M,$
 hence  it  is contained in the   mean convex side of $M.$ This implies  that the mean curvature of  $P$ 
at $\phi(p)$ points towards $M\times\{0\}.$ \\  
 From   equation \eqref{cmc30} in  Appendix \ref{cmcequation}, the mean curvature  $H_P$ of $P$ is given by \begin{equation}
\label{H-P}
 nH_P = nH + \Delta\phi + O(\rho^{\alpha})
 \end{equation}

  where the term $ O(\rho^{\alpha})$  (with $0<\alpha<1$)  
 converges  uniformly to zero as $\rho$ tends to zero 
provided  $M$ and $\mathcal N$ have uniform curvature bounds.\\
 Since  the mean curvature of $P$ points towards $M\times\{0\},$ the hypothesis    $H\geq \frac{\varepsilon}{n}$ implies  $H_P \leq -\frac{\varepsilon}{n}.$   Therefore,   for $\rho$ sufficiently small, one has
  
   \begin{equation}
\label{tube2}
   \Delta\phi\leq  -2\varepsilon
   \end{equation}

 Now we construct a function $\psi_R$ in terms of the distance function from a fixed point $p_0$ of $M,$ such that 
 $\Delta \psi_R\geq -\varepsilon$ on $B_{p_0}(R)\cap\Omega.$  
 Let $r$ be the distance function from $p_0 \in M,$ and consider the radial test function   $\psi_R (x)= f_R\circ r(x) $ where 
 
 \begin{equation}
f_R (r) = \left\{
\begin{array}{ll}
 \beta\left(  1 - \left(\frac{r}{R}\right)^2\right)\qquad   &   \quad \forall \  r \leq R, \\ 
0 \qquad   & \quad \forall \  r \geq  R. \\ 
\end{array}
\right.
\end{equation}
with  $\beta = \rho -\delta,$ for small positive $\delta.$  
 Notice that $\psi_R$ vanishes  on $\partial   B_{p_0}(R) \cap \Omega$ 
and  $\psi_R \leq \phi$ on $B_{p_0}(R) \cap \partial\Omega$, since $\phi|_{\partial \Omega} = \rho$.  Therefore $\psi_R \leq \phi $ on 
 $\partial (\Omega \cap B_{p_0}(R)) $. Since $M$ has bounded curvature, there exists $k>0,$ such that $Ric_M\geq -(n-1)k^2.$
By standard comparison theorems (see for instance  \cite{SY}) one has 
 
 \begin{equation}
 \label{tube3}
 \Delta r  \leq \frac{n-1}{r} (1+kr)
 \end{equation} 
where the inequality holds  outside the cut-locus of $M$ and   holds   in the weak sense at any point of $M.$  
 Using inequality    \eqref{tube3} and  that  $ \Delta f_R(r) = f_R'(r)\Delta r  + f_R''(r) |\nabla r|^2,$ one has   
\begin{equation}\label{tube4}
\Delta \psi_R \geq  -\frac{2\beta}{R^2}(n+ (n-1)kR).
 \end{equation}
 
For $R$  large,  $\Delta \psi_R \geq -\varepsilon,$ as we wished.  This last inequality with   inequality \eqref{tube2} yield    
$\Delta \phi \leq\Delta \psi_R $   on $  B_{p_0}(R)\cap \Omega,$ for  $R$ large. Then,  by    Corollary \ref{basic2},  
    $\phi \geq \psi_R $ on  $  B_{p_0}(R)\cap \Omega,$ for $R$ large. 
  Letting  $R \rightarrow \infty, $  we obtain  $\phi \geq \beta$ on $\Omega$. Therefore $\phi \geq \rho-\delta$ in $\Omega$ for any $\delta>0.$ 
  Thus  $\phi \geq \rho$ in $\Omega.$ 
This is a contradiction, hence   $T^+(\rho)$ is embedded.

(2) The proof is the same as in (1), except than for the choice of the test function $\psi_R.$ 
Without loss of generality, we can assume that  $\varepsilon <1$ and $\rho <1$. 
 Let  $R_1>1 $ be  such that $K\subset B_{p_0}(R_1-1)$  and let $\sigma := \underset{p\in \partial B_{p_0}(R_1-1)}{\inf}\phi(p)$ (notice that 
 $\sigma \leq \rho \leq 1$). 
Let  
  $\Omega' = \Omega \cap (M\setminus B_{p_0}(R_1))$ and let $r(x)$ be the distance function in $M$  from any fixed point $p_0$ in $\Omega'.$
 Define   $R$ to be the distance  from  $p_0\in\Omega'$   to $B_{R_1-1}$ (notice that $R>1$).
 We define a radial test function  $\psi_R = g_R\circ r$ where 
 
  \begin{equation}
  \label{tube5}
g_R (r) = \left\{
\begin{array}{ll}
 \beta'\left(  1 - \left(\frac{r}{R}\right)^2\right)\qquad   &   \quad \forall \  r \leq R, \\ 
0 \qquad   & \quad \forall \  r \geq  R. \\ 
\end{array}
\right.
\end{equation}
where the constant  $\beta' $ equals   $ \frac{\varepsilon}{2(n+(n-1)k)}
 (\sigma-\delta), $   where $\delta  $ is  small positive number 
and $k$ is such that  $Ric_M\geq -(n-1)k^2.$
 Notice that $\psi_R$ vanishes  on $\partial   B_{p_0}(R) \cap \Omega'$ 
and  $\psi_R \leq \phi$ on $B_{p_0}(R) \cap \partial\Omega'.$
 Therefore $\psi_R \leq \phi $ on 
 $\partial (\Omega' \cap B_{p_0}(R)). $ 
 
 Using inequality \eqref{tube3}, by a straigtforward computation one has
 
  \begin{equation}\label{tube6}
\Delta \psi_R \geq  - \varepsilon(\sigma-\delta)
\end{equation}

Since $0<\sigma -\delta <1$,  $\Delta \psi_R \geq -\varepsilon,$ as we wished.  This last inequality together  with   inequality \eqref{tube2} yield    
$\Delta \phi \leq\Delta \psi_R $   on $  B_{p_0}(R)\cap \Omega',$. Then,  by    Corollary \ref{basic2},  
    $\phi \geq \psi_R $ on  $  B_{p_0}(R)\cap \Omega',$ for $R$ large.  

 In particular $\phi(p_0) \geq \beta'$  for any $p_0\in  \Omega\setminus B_{p_0}(R_1)$
 This proves that the tube $T^+(\beta')(  \Omega\setminus B_{p_0}(R_1)$ is embedded.

 \end{proof}
   
\section{Caccioppoli's inequality for  constant mean curvature hypersurfaces with finite index}   
 \label{caccioppoli}

In this section we recall some results obtained in \cite{INS} and needed in the sequel. We assume that ${\mathcal N}$ is an  orientable Riemannian manifold with bounded sectional curvature. Moreover  let $M$  be an orientable  hypersurface  immersed in ${\mathcal N}$ and assume that $M$ 
has  constant mean curvature. When the mean curvature is non zero,  we  orient $M$ by its mean curvature vector $\vec H.$ Then $\vec H=H\vec \nu$   with $H$  positive  and $\vec\nu$ a unit normal vector.    When the mean curvature is zero, we choose once for all an orientation $\vec\nu$ on $M.$ We denote by  $\varphi$ the length of the    traceless part of the second fundamental form $A$ i.e. $\varphi := |A - Hg|.$
In the present article we will need the 
following Caccioppoli's inequalities (Theorems 5.1 and 5.4 of  \cite{INS}) 
 and a reverse H\"{o}lder inequality (Theorem 4.2  \cite{INS}).

\begin{theorem} [{\bf Caccioppoli's inequality of type I}]
\label{caccio-theo}
Let $M$ be a complete hypersurface immersed with constant mean curvature $H$ in a manifold $\mathcal N.$  Assume $M$  has finite index. 
Then, there exist a compact subset $K$ of $M$ ( which is empty if $M$ is stable)   and   constants $\beta_1,$  $\beta_2,$  $\beta_3,$ such that for every $f\in C_0^{\infty}(M\setminus K)$ and $x\geq 1$
 \begin{equation}
\label{i3}
\beta_1\int_{M\setminus K} f^{2x+2}\varphi^{2x+2} \leq \beta_2 \int_{M\setminus K}  |\nabla f|^{2x+2}+\beta_3\int_{M\setminus K}  f^{2x+2}.
\end{equation}
Moreover  the constant $\beta_1$ is positive  if and only if $x\in[\left.1,1+\sqrt{\frac{2}{n}}\right),$  \end{theorem}

 \begin{theorem}
 \label{reduction-exponent}
 Let $M$ be a complete hypersurface immersed with constant mean curvature $H$ in a manifold with constant sectional curvature $c.$  Assume $M$ has finite index. Then there exists  a compact subset $K$ of $M$  ( which is empty if $M$ is stable)  and a positive  constant $\mathcal S$ such that for any $x\in\left.[1,1+\sqrt{\frac{2}{n}}\right)$
\begin{align}
\label{reduction-ineq}
\int_{M\setminus K} \varphi^{2x+2}\leq {\mathcal S} \int_{M\setminus K} \varphi^{2x}
\end{align}
\end{theorem}

It is  worthwhile  to note  that, if $M$ is stable,  using a  suitable test function  in inequality (58)  of  the proof of Theorem 4.2 in \cite{INS},  
one can deduce

\begin{equation}
\label{n-red-exp}
\int_{B_{p_0}(R)}\varphi^{2x+2}\leq{\mathcal S}\int_{B_{p_0}(R+1)}  \varphi^{2x}
\end{equation}
 
 where  $p_0$ is a fixed point of $M$ and $\mathcal S$ is a positive constant. 
 
 Before stating  the  next Theorem, we need two define new notations:

\nero  For $\gamma=\frac{n-2}{n},$ $\mu=\frac{n^2}{4(n-1)}, $   let $g$ be the following function
\begin{equation}
\label{function-g}
g_n(x)=\frac{(2x-\gamma)^2-x^4}{(2x-\gamma)^2-\mu x^4}
\end{equation}

\nero Let  $x_2$ the following real number

\begin{equation}
\label{roots}
  x_2=\frac{2\sqrt{n-1}}{n} \left(1+\sqrt{1-\frac{n-2}{2\sqrt{n-1}}}\right)
\end{equation}

\begin{theorem} [{\bf Caccioppoli's inequality of type III - $H\not=0$}]
\label{caccio-H}
Let $M$ be a complete hypersurface immersed with constant mean curvature $H\not=0,$ in a manifold with  constant   curvature $c.$ Assume  $M$ has finite index  and $n\leq 5.$   
Then there exist  a compact subset $K$ in $M$  and a constant $\gamma$ such that,  for any $f\in C^{\infty}_0(M\setminus K)$  
\begin{equation}
\label{j-bis}
\gamma\int_{M\setminus K} f^2 \varphi^{2x}\leq {\mathcal D}\int_{M\setminus K} \varphi^{2x}|\nabla f|^2
\end{equation}
provided either 
 (1)  $c=0$ or $1,$  $x\in[1,x_2)$ or 
(2)   $c=-1,$  $\varepsilon>0,$  $x\in[1, x_2-\varepsilon],$ $H^2\geq g_n(x).$
\end{theorem}

   
\section{Entropies}
   \label{entropy-section}
Let $\mathcal N$ be a complete, noncompact Riemannian manifold. In this section, except when it is indicated, there is no hypothesis on the curvature of $\mathcal{N}$.\\
In this  Section, we deal with the exponential growth of various functionals on $\mathcal{N}.$ As we explained  in the Introduction, 
we will call such exponential  growths, {\it  entropies}.

The most important one is the entropy associated to the volume of geodesic balls in ${\mathcal N}$ (see for instance \cite{Br}, \cite{Br1}).

\begin{definition}  Let $w$ be a positive non-decreasing function.  The entropy of $w$  is by definition
 
 \begin{equation}\label{entropy}
\mu_w := \underset{r\lra\infty}{\limsup} \left(\frac{\ln w(r)}{r}\right).  
 \end{equation}
 
\end{definition}

We say that     the function $w$ has {\it subexponential growth}    if its entropy is zero. It is   worth  noting   that  $\mu_w =0$    is equivalent to

\begin{equation}\label{sub-exponential}
\limsup_{r\lra\infty} \frac{ w(r)}{e^{\alpha r}}= 0, 
  \quad  \forall \alpha >0. \end{equation}
  
We say that    $w$ has {\it  exponential growth  } if  its entropy is positive. 

We observe  that    having a  subexponential  growth is a  assumption weaker than   being   bounded by a polynomial of any degree (for instance  $w(r)=e^{r^{\beta}},$ $\beta<1$ has subexponential growth).
Let $B_{\sigma}^{\mathcal N}(R)$ be a geodesic ball in $\mathcal N,$ of radius $R,$ centered at a fixed point   
 $\sigma\in \mathcal N$  and denote  by 
$|B_{\sigma}^{\mathcal N}(R)|$ its  volume. When $w(R)=|B_{\sigma}^{\mathcal N}(R)|,$ the entropy of $w$ is called the {\em volume entropy} of 
$\mathcal N$ and it is denoted by 

\begin{equation}
\label{volume-entropy}
\mu_{\mathcal N}:= \underset{R\lra\infty}{\limsup} \left(\frac{\ln |B_{\sigma}^{\mathcal N}(R)| }{R}\right).
\end{equation}

Using the distance inequality, one can easily check  that the volume entropy does not depend on the center $\sigma$ of the balls.\\
\begin{remark}
In   the definition of entropies, one can take $\liminf$ instead of $\limsup$. 
\end{remark}
Let $M$ be  a  complete, noncompact manifold immersed in $\mathcal N.$ Let $g$ (respectively $A$)  be the induced metric (respectively the second fundamental form)  of $M$ and consider the traceless second fundamental form of $M,$ i.e.  $\phi=A- Hg,$ where $  H$ is the mean curvature   of $M.$ 

In addition to the volume entropy of $M,$ we will be interested in  the following entropies. 

\nero  {\em  Extrinsic volume entropy}  of $M.$  It is denoted by $\mu_{M}^{\mathcal N}$ and  is obtained by replacing, in  the definition of the volume entropy,  the volume of the intrinsic balls of $M$ by  the volume in $M$ of $B_{\sigma}^{\mathcal N}(R)\cap M_{\sigma}$  where $M_{\sigma}$ is the connected component of $M$ containing ${\sigma}$.\\
As for  the volume entropy, the  extrinsic volume entropy doesn't depend on the choice of ${\sigma}$.\\
\nero {\em Total $p$-curvature entropy} of $M.$ It is denoted by  $\mu_{{\mathcal T}_p},$  for any $p>0$  and it is the entropy of the    {\it total $p$-curvature} ${\mathcal T}_p$  defined by   
\begin{equation}
\label{p-entropy}
{\mathcal T}_{p}(R)=\int_{B_{\sigma}(R)} |\phi|^{p}
\end{equation}

 There are  important relations between the volume entropy of  a manifold ${\mathcal N}$ and the bottom of its  {\em essential spectrum}.  
Let $\Delta$ be the Laplacian  on $\mathcal N,$ then the bottom of the  spectrum $\sigma(\mathcal N)$ of $-\Delta$ is

\begin{equation}
\label{lambda-zero}
\lambda_0({\mathcal N})=\inf\{\sigma(\mathcal N)\}=\underset{f\not=0}{\underset{f\in C_0^{\infty}({\mathcal N})}{\inf}}\left(\frac{\int_{\mathcal N}|\nabla f|^2}
{\int_{\mathcal N} f^2}\right).
\end{equation}

The bottom of the essential spectrum $\sigma_{ess}(\mathcal N)$ of $-\Delta$ is 

\begin{equation}
\label{lambda-ess}
\lambda^{ess}_0({\mathcal N})=\inf\{\sigma_{ess}(\mathcal N)\}=\underset{K}{\sup} \lambda_0({\mathcal N}\setminus K)
\end{equation}

where $K$ runs through all compact subsets of $\mathcal N.$

Another invariant related to  the spectrum is the Cheeger isoperimetric constant. Recall that the Cheeger constant $h_{\mathcal N}$ of a Riemannian manifold $\mathcal N$ is defined as
$h_{\mathcal N}=\underset{\Omega}{\inf}\frac{|\partial \Omega|}{|\Omega|},$  where $\Omega$ runs  over all compact domains of $\mathcal N$ with piecewise smooth boundary $\partial \Omega$ .

J. Cheeger \cite{Ch}
 and R. Brooks   \cite{Br} proved the following important comparison result  between the bottom of the essential spectrum, the volume entropy and the Cheeger constant:
 
\begin{theorem}[{\bf Brooks-Cheeger's  Theorem}]
\label{brooks}
  If $\mathcal N$ has infinite volume, then 

\begin{equation}
\label{brooks-ineq}
\frac{h_{\mathcal N}^2}{4}\leq \lambda_0(\mathcal N)\leq \lambda_0({\mathcal N}\setminus K)\leq \lambda_0^{ess}({\mathcal N})\leq \frac{\mu_{\mathcal N}^2}{4}
\end{equation}
where $K$ is any compact subset of $\mathcal N$ and 
 $h_{\mathcal N}$ is the  Cheeger constant of  $\mathcal N.$ 
\end{theorem}

\begin{remark}
\nero The upper bound in \eqref{brooks-ineq} is still true if one replaces $\limsup$ by $\liminf$ in the definition of the entropy (see \cite{Hi}, \cite{O}).\\
\nero When $\mathcal N$ has finite volume, we have $h_{\mathcal N}=\mu_{\mathcal N}=\lambda_0(\mathcal N)=0$ and Theorem \ref{brooks} becomes trivial.  In the finite volume case,  R. Brooks introduced a modified $h$ and modified $\mu$  and was able to  obtain a similar but nontrivial result  (see \cite{Br1}).\\
\nero  We observe that  the infiniteness of the volume follows, for example, from the existence of a
Sobolev inequality  (see \cite{Ca} and \cite{He}).  Another important case  of infiniteness of volume, useful for  us, is that where $M$ 
is a complete noncompact submanifold with bounded mean curvature of a manifold with bounded geometry (see \cite{Fr}).\\
 A result analogous to that of Theorem \ref{brooks} for  the $p$-Laplacian is known. In fact, B. P. Lima, J. F.  Montenegro and N.L. Santos \cite{LMS}, adapting Theorem 2 in \cite{Br} proved a Brooks type result for the p-Laplacian, and the generalization of the lower bound in terms of the Cheeger constant, in the case of p-Laplacian, was obtained by D.  Lefton and L.  Wei \cite{LW} (see also \cite{Ma} for the case where the manifold is compact) .
\end{remark}

It is easy to compare the volume entropy  with the extrinsic volume entropy. Indeed we have the following result.
  \begin{proposition} 
\label{extrinsic-theo}
 Let $M$ be a complete, noncompact manifold. Then 
\begin{equation}
\label{extrinsic-ineqs}
h_{M} \leq \mu_{M}\leq \underset{\mathcal N}{\inf}\mu_M^{\mathcal N}:=j_{M}
\end{equation}
where $\mathcal N$ runs in the set of all manifolds  $\mathcal N$ in  which  $M$ admits an isometric immersion.
\end{proposition}

\begin{proof} 

The first inequality of  \eqref{extrinsic-ineqs} is an immediate consequence of Theorem \ref{brooks}. We wrote it, in order to emphasize that it is independent of the spectrum. For the sake of completeness, we give a very short proof of it. In fact, it suffices to observe, that for any $x \in M$ and for any $R>0$, we have 
$$ \frac {\vert \partial B_x^{M}(R)\vert}{\vert B_x^{M}(R)\vert}=\frac{d}{dr}\vert_{r=R} \left(\ln \vert B_x^{M}(R)\vert \right) \ge h_{M}$$
which gives after integration and after taking the limit for $R \to \infty$, the desired inequality $h_{M} \leq \mu_{M}$.\\

For the second inequality, it suffices to observe that for any  $x\in M$, $R>0$ and any $\mathcal N$  satisfying the hypothesis,  one has  
\begin{equation*}
B_{x}^M(R)\subset M\cap B_{x}^{\mathcal N}(R)
\end{equation*}
then 
\begin{equation*}
|B_{x}^M(R)|_M\leq |M\cap B_{x}^{\mathcal N}(R)|_M,
\end{equation*}
where $|\cdot|_M $ is the volume in $M.$ 
Inequality \eqref{extrinsic-ineqs} follows from the definition of the entropies.

\end{proof}

\begin{remark} A comparison similar to  \eqref{extrinsic-ineqs} holds for the  intrinsic and extrinsic entropies of the $p$-total curvatures, that can be obviously defined.

\end{remark}

It is  worth  stating the following immediate consequence of Theorems \ref{brooks}and   Proposition  \ref{extrinsic-theo},
that gives an extrinsic upper bound of $\lambda_0(M).$ 

\begin{corollary}
\label{ext-brooks-theo}
Let $M$ be a complete, noncompact manifold with infinite volume, then
\begin{equation}
\label{ext-brooks-ineq}
\frac{h_M^2}{4}\leq\lambda_0(M)\leq  \lambda_0^{ess}({M})\leq \frac{j_{M}^2}{4}
\end{equation}
\end{corollary}

\begin{remark}
All  the previous  results can be adapted to the case where $M$ has  finite volume, using the modified volume entropy and Cheeger constant introduced by Brooks \cite{Br1}.  
\end{remark}

Now we recall some classical estimates for  the volume entropy and the  Cheeger constant  in terms of bounds on curvatures.
A first result is the following well known consequence of the Bishop volume comparison theorem.\\

\begin{lemma}(see  \cite{Ka})
\label{compar1}
Let $M$ be a complete Riemannian manifold of dimension $m$ and such that 
$Ric_{M} \geq -(m-1)\,k^{2}$ (for some constant $k$).  
 Then we have:
\begin{equation}
\label{entr-ricci}
\mu_M\leq (m-1)\,k 
\end{equation}
\end{lemma}

The proof of next Lemma is based on the comparison theorems for the hessian of the distance function,
\begin{lemma} \label{yau} (see \cite{Y})
Let $M$ be a complete simply connected Riemannian manifold of dimension $m$ with sectional curvature bounded from above by $-k^{2}.$
Then we have:
\begin{equation}
h_M\geq (m-1)\,k 
\end{equation}
\end{lemma} 
The result of Lemma \ref{yau} was extended to minimal submanifolds of such simply connected manifolds by  J. Choe and R. Gulliver \cite{CG}. The proof in \cite{CG} (in particular Lemmas 7 and 8) can be easily adapted to submanifolds of bounded mean curvature in order to give the following result:
\
\begin{lemma}
\label{compar2}
Let $\mathcal{N}$ be an $n$-dimensional complete simply connected Riemannian manifold with sectional curvature bounded from above by $-k^{2}$  (for some positive constant $k$). Let $M$ be a complete noncompact 
submanifold of $\mathcal{N}$ with bounded mean curvature satisfying $|H| \leq a$,  then we have:
\begin{equation}
\label{cheeg-sect}
h_M + a\geq (n-1)\,k 
\end{equation}
\end{lemma}

\begin{remark}  
\nero  The estimate given in (\ref{cheeg-sect}) is nontrivial only for $(n-1)\, k > a$.

\nero   By Lemma \ref{compar1} and the upper bound in  \eqref{extrinsic-ineqs}, if $Ric_{M} \geq -(m-1)\,k^{2}$, then we obtain the well known Mac Kean estimate: 

$$ \lambda_{0}(M) \leq \frac{(m-1)^2 k^2}{4}.$$\\

 \end{remark}
 
   Many recent results in the literature can be deduced or generalized immediately from our results. Let us give some examples of such consequences. 
 \begin{enumerate}
 
  \item   B. P. Lima, J.F. Montenegro and N.L. Santos (see \cite{LMS}) obtain a generalization of Brooks upper bound for the p-Laplacian. More precisely, if one denotes by  $ \lambda_{0,p}(M)= \underset{f\not=0}{\underset{f\in C_0^{\infty}(M)}{\inf}}\left(\frac{\int_{M}|\nabla f|^p}
{\int_{M} f^p}\right)$ and $ \lambda_{0,p}^{\rm ess}(M)=\underset{K}{\sup}\, \lambda_{0,p}({M}\setminus K)$, where $K$  runs through all compact subsets of $\mathcal N,$ they prove that
$$ \lambda_{0,p}(M) \le \lambda_{0,p}^{\rm ess}(M) \le \left(\frac{\mu_{M}}{p}\right)^{p}.$$

 Using  inequality \eqref{extrinsic-ineqs} in the previous inequality,  we deduce immediately  an upper bound  for 
$ \lambda_{0,p}(M)$
 in terms of the extrinsic volume entropy, which gives a p-Laplace version of  Corollary \ref{ext-brooks-theo}. \\ 
 
\item  Under the same assumption as in Lemma \ref{compar2}, if $a<(n-1)k,$ Corollary \ref{ext-brooks-theo} yields $\lambda_{0}(M) \geq \frac{((n-1)\,k-a)^{2}}{4}.$  Hence we obtain the result by L.F.  Cheung and P.F. Leung \cite{CL}  and G.P.  Bessa and J. F. Montenegro \cite{BM1}.

\item In \cite {LW}, L. Lefton and D. Wei obtained the following generalization of Cheeger inequality (the first inequality of \eqref{brooks-ineq}) to the first eigenvalue of $p$-Laplace operator, that is  
$$\left(\frac{h_{M}^2}{p}\right)^{p}\leq \lambda_{0,p}(M)$$ 
 The proof  in \cite{LW} is done in the Euclidean case, but  it  can be easily adapted to a manifold. This inequality combined with Lemma \ref{yau} gives the generalization to p-Laplacian of the Mc Kean inequality obtained in \cite {LMS}, 
and combined with Lemma \ref{compar2} gives  the inequality 
$$\lambda_{0,p}(M) \geq \frac{((n-1)\,k-a)^{p}}{p^{p}}$$
which generalizes the aformentioned results of \cite{CL} and  \cite{BM1} \\
\item Theorem A  in \cite{GP} and the main Theorem   in \cite{G} can be  immediately deduced  from 
our results. Let us be more precise. In \cite{GP}, V. Gimeno and V. Palmer  proved  that if $M$ is a minimal $m$-dimensional submanifold properly immersed in a simply connected manifold $\mathcal N,$ such that 

\begin{equation}\label{Gimeno-Palmer}
sec({\mathcal N})\leq b\leq 0, \  \   \underset{R}{\sup}\frac{|B_{\sigma}^{{\mathcal N}}(R)\cap M|}{|{\mathcal B}_{\sigma}^m(R)|}<\infty
\end{equation}

where 
${\mathcal B}_{\sigma}^m(R) $ is the geodesic ball centered at $\sigma$ of  ${\mathbb H}^m(b)$ (the space form of sectional curvature $b$), then 

\begin{equation}\label{Gimeno-Palmer-result}
h_M\leq(m-1)\sqrt{-b}.
\end{equation}

As the volume entropy of ${\mathbb H}^m(b)$ is $(m-1)\sqrt{-b},$ the second  hypothesis in \eqref{Gimeno-Palmer}
 implies that $\mu_M^{\mathcal N}\leq (m-1)\sqrt{-b}.$ Then, Corollary \ref{ext-brooks-theo} yields \eqref{Gimeno-Palmer-result} without the minimality hypothesis. \\
 
 In \cite{G}, V. Gimeno proves that under the same  assumptions \eqref{Gimeno-Palmer} (as in \cite{GP}),   one has  in addition the equality 
 $\lambda_{0}(M)=-\frac{(m-1)^{2}b}{4}$. This is an immediate consequence of Corollary \ref{ext-brooks-theo} and Lemma \ref{compar2} which give in this case: 
$$ h_{M}=\lambda_{0}(M)=\lambda_0^{ess}({M})=-\frac{(m-1)^{2}b}{4},$$

again without any minimality assumption.

\end{enumerate}

It is possible to obtain a stronger result   about the comparison between the volume of  extrinsic balls and the volume of balls of the ambient space, 
by applying Proposition \ref{tubeplonge}. 

  \begin{corollary}\label{volume-estimate}  Let $M$ be a noncompact, complete, hypersurface of bounded curvature, properly embedded in a   simply-connected   manifold ${\mathcal N}$ with bounded curvature.  Assume  $M$    has  mean curvature  bounded away from zero. 
 Then, one has:
    \begin{equation}
   \label{volume-tube-hyp}
   | B_{\sigma}^{ \mathcal{N}}(R) \cap M |_{M} \leq c|B_{\sigma}^{\mathcal{N}}(R) | _{\mathcal{N}}  
     \end{equation}
   
   where $c$  is a constant depending on  the radius $\rho$   of the embedded  half-tube given in Theorem \ref{tubeplonge}  and on the curvature of $M$  and $|\cdot|_M$ (respectively $|\cdot|_{\mathcal N}$) 
   denotes the volume in $M$ (respectively in $\mathcal N$). 
 
 \end{corollary}

\begin{proof}

By Proposition \ref{tubeplonge} there exists $\rho$ such that the half tube $T^+ (\rho)$ is embedded.  
Hence, by Weyl's formula, the volume of   $T^+ (\rho)\cap B_{\sigma}^\mathcal N(R)$  satisfies 

\begin{equation}
\label{weyl}
| B_{\sigma}(R)^\mathcal N \cap M |_{M} \leq c|T^+ (\rho)\cap B_{\sigma}^\mathcal N(R)|
\end{equation}

where $c$ is a constant depending only on $\rho,$ the curvature of $M$ and $\mathcal N$ (see  for instance \cite{Gr} or \cite{NS}). 
As  $|T^+ (\rho)\cap B_{\sigma}^\mathcal N(R)|\leq |B_{\sigma}^\mathcal N(R)|,$ inequality  \eqref{volume-tube-hyp} follow from inequality \eqref{weyl}.
\end{proof}

 Then, we are able to establish a relation between the intrinsic and the extrinsic volume entropies.
 
 \begin{corollary} \label{volume-estimate2}  Let $M$ be a noncompact, complete, hypersurface of bounded curvature, properly embedded in a   simply-connected   manifold ${\mathcal N}$ with bounded curvature.  Assume  $M$    has  mean curvature  bounded away from zero. 
Then 
$$\mu_M \leq \mu^{\mathcal{N}}_M  \leq \mu_{\mathcal{N}}$$ 
and hence  
$$\lambda_0^{ess} (M) \leq \frac{\mu^2_{\mathcal{N}}}{4}$$
 \end{corollary}
\begin{proof}
Corollary \ref{volume-estimate}  gives the first inequality. For the second inequality, we use Corollary 
\ref{ext-brooks-theo} since $M$ has infinite volume (see  \cite{Fr}).
\end{proof}

In particular, we have the following result.
\begin{corollary}\label{zero-ent}
 Let $M$ and $\mathcal N$ be
 manifolds of bounded curvature such that $M$  is  a noncompact, complete, hypersurface  properly embedded in  ${\mathcal N}$. 
 Assume $\mathcal N$ is simply connected. If $ \mu_{\mathcal{N}}=0$, then $\mu_M=\mu^{\mathcal{N}}_M=h_{M} =\lambda_{0}^{ess}(M)=\lambda_{0}(M)=0$.
\end{corollary}
\begin{remark}
\nero  Let $M$ be as in Corollary \ref{zero-ent}. Then $0$ belongs to $\sigma_{ess}$ but  $0$ is not an eigenvalue of $\Delta,$ 
because $M,$ having infinite volume, does not carry any $L^2$ harmonic function.

\nero  Lemma \ref{compar1} guarantees that,  if  $Ric_{\mathcal{N}} \geq 0,$ then $ \mu_{\mathcal{N}}=0$. More generally, $\mu_{\mathcal N}=0$ when the manifold $\mathcal{N}$ satisfies the doubling volume property since, as in the case of nonnegative Ricci curvature, the volume of balls grows polynomially.\\
\end{remark}

There is a strong relation between Caccioppoli's inequality and the entropy of the $p$-total curvature.

The following general Theorem will be used to show such relation in Section \ref{caccio-appli}.  
 
\begin{theorem}
\label{caccio-entropy}  Let $w \in L^1_{loc}( M ; \mathbb{R}^+)$ and $W(R) := \int_{B_{\sigma}(R)}  w .$    Assume that the entropy of $W$   is zero. 
Then we have the following results. 

(1) If , for some positive constant $C,$ $w$ satisfies 
 
\begin{equation}
\label{har}
\int_M w\psi^2  \leq C\int_M w|\nabla \psi|^2 , \qquad \forall \psi \in W^{1,2}_0(M),
\end{equation}
  then $w \equiv 0.$
  
 (2)  If, for some positive constant $C,$ and some compact subset $K,$  $w$ satisfies 
 
\begin{equation}
\label{har1}
\int_{M\setminus K} w\psi^2  \leq C\int_{M\setminus K} w|\nabla \psi|^2 , \qquad \forall \psi \in W^{1,2}_0(M\setminus K),
\end{equation}
  then $\int_M w <\infty.$ 
\end{theorem}

\begin{proof}
(1) Fix a point $\sigma$ in $M$ and  for any $x\in M,$ let $r(x)=d(x,\sigma)$  be the distance  between $x$ and $\sigma.$ Let $\alpha>0$ and 
let 
$\psi\in W^{1,2}_0(M)$ be the radial function such that  $\psi(x)= e^{-\alpha r(x)} - e^{-\alpha R}$ on $B_{\sigma}(R)$ and $\psi\equiv 0$ on $M\setminus 
B_{\sigma}(R).$

 Applying  inequality \eqref{har} to  $\psi$   we obtain 

 \begin{equation}\label{ine1}
\int_{B_{\sigma}(R)} w  \left(e^{-\alpha r} - e^{-\alpha R}\right)^2 \leq
C\alpha^2\int_{B_{\sigma}(R)} w  e^{-2\alpha r}.
\end{equation}

First, we  prove that the   right hand side of \eqref{ine1}  is bounded independently of $R$. 
 By the co-area formula  (see for instance  Formula 3.8 in \cite{Mo} or \cite{SY})

\begin{equation}\label{ine2}
\int_{B_{\sigma}(R)} w  e^{-2\alpha r}     
=\int_0^R \left(\int_{S_r} w   d\sigma_r \right) e^{-2\alpha r}dr      
\end{equation}

where $d\sigma_r$  is the volume element of $S_r=\partial B_{\sigma}(R)$.  Then,  integrating  
by parts,  equality \eqref{ine2}, one has

\begin{equation}
\label{ine3}
\int_{B_{\sigma}(R)} we^{-2\alpha r}     
 =  W(R)e^{-2\alpha R}  +2\alpha  \int_{0}^R  W(r)e^{-2\alpha r} dr  
 \end{equation}

As the entropy of $W$  is zero, we deduce  that the first term of the right hand side of \eqref{ine3} tends to  zero for $R\longrightarrow\infty.$ On the other hand, as 
$We^{-\alpha r}$ is bounded, the second term of the right hand side of \eqref{ine3} converges for $R\longrightarrow \infty.$
 Hence   the right hand side  of inequality \eqref{ine1}  is  bounded independently of $R$ and the left hand side is bounded as well.
 
Therefore, for any compact $K\subset M,$ there exists $R$ so that

 \begin{equation}\label{ine4}
\int_{K} w  \left(e^{-\alpha r} - e^{-\alpha R}\right)^2 \leq
C\alpha^2\int_{M} w e^{-2\alpha r}.
\end{equation}

Letting $R\longrightarrow\infty$ in \eqref{ine4} gives,  for any compact $K\subset M$ 

\begin{equation}\label{ine5}
\int_{K} w  e^{-2\alpha r}\leq C\alpha^2\int_{M} w  e^{-2\alpha r}.
\end{equation}

Hence

\begin{equation}\label{ine6}
\int_{M} w  e^{-2\alpha r}\leq C\alpha^2\int_{M} w e^{-2\alpha r}.
\end{equation}

In order to conclude the proof of (1), we choose $\alpha$ such that $C\alpha^2<1.$

(2) Fix a point $\sigma$ in $M$ and  for any $x\in M,$ let $r(x)=d(x,\sigma)$  be the distance  between $x$ and $\sigma.$ Let $0<R_0<R_1<R_2<R_3,$  
$\alpha>0$ and  let 
$\phi\in W^{1,2}_0(M)$ be  the radial function such that $\phi\equiv0$ on ${B_{\sigma}(R_0)},$ $\phi$ is linear on 
${B_{\sigma}(R_1)}\setminus {B_{\sigma}(R_0)},$  $\phi\equiv0$ on
${B_{\sigma}(R_2)}\setminus {B_{\sigma}(R_1)},$ $\phi (r) = \frac{e^{-\alpha(r-R_2)} - e^{-\alpha(R_3-R_2)}}{1- e^{-\alpha(R_3-R_2)}}$ on 
${B_{\sigma}(R_3)}\setminus {B_{\sigma}(R_2)}$ and $\phi\equiv0$ on $M\setminus {B_{\sigma}(R_3)}.$ 

Replacing $\phi$ in \eqref{har1} yields 

\begin{equation}
\label{ine7}
\int_{{B_{\sigma}(R_1)}\setminus {B_{\sigma}(R_0)}}w(\phi^2-C|\nabla\phi|^2)+\int_{{B_{\sigma}(R_2)}\setminus {B_{\sigma}(R_1)}}w+
\int_{{B_{\sigma}(R_3)}\setminus {B_{\sigma}(R_2)}}w\phi^2\leq
C\int_{{B_{\sigma}(R_3)}\setminus {B_{\sigma}(R_2)}}w|\nabla\phi|^2
\end{equation}

As the third term in the left hand side of \eqref{ine7} is positive, we can remove it and obtain

\begin{equation}
\label{ine8}
\int_{{B_{\sigma}(R_1)}\setminus {B_{\sigma}(R_0)}}w(\phi^2-C|\nabla\phi|^2)+\int_{{B_{\sigma}(R_2)}\setminus {B_{\sigma}(R_1)}}w\leq
C\int_{{B_{\sigma}(R_3)}\setminus {B_{\sigma}(R_2)}}w|\nabla\phi|^2
\end{equation}

We claim that the right hand term of \eqref{ine8} is bounded independently of $R_3.$ 

A straightforward computation shows that there exists  a constant $C_1$ independent of $R_3,$ such that on ${B_{\sigma}(R_3)}\setminus {B_{\sigma}(R_2)}$

\begin{equation}
\label{ine9}
|\nabla\phi(r)|^2=\frac{\alpha^2}{(1-e^{-\alpha(R_3-R_2)})^2}e^{-2\alpha(r-R_2)}\leq C_1e^{-2\alpha(r-R_2)}
\end{equation}

Then, in order to prove our claim, one needs to bound    $\int_{{B_{\sigma}(R_3)}\setminus {B_{\sigma}(R_2)}}we^{-2\alpha(r-R_2)}$ independently of $R_3.$

This goes exactly as  in the proof of the boundedness of the right hand side of \eqref{ine1}. For the sake of completeness we do it.
By the co-area formula  (see for instance  Formula 3.8 in \cite{Mo} )

\begin{equation}\label{ine10}
\int_{{B_{\sigma}(R_3)}\setminus {B_{\sigma}(R_2)}} w  e^{-2\alpha (r-R_2)}     
=\int_{R_2}^{R_3} \left(\int_{S_r} W   d\sigma_r \right) e^{-2\alpha (r-R_2)}dr      
\end{equation}

where $d\sigma_r$  is the volume element of $S_r=\partial B_{\sigma}(R)$.  Then,  integration  
by parts   of equality \eqref{ine10} yields

\begin{equation}
\label{ine11}
\int_{{B_{\sigma}(R_3)}\setminus {B_{\sigma}(R_2)}} w  e^{-2\alpha (r-R_2)}      
 =  W(R_3)e^{-2\alpha( R_3-R_2)}-W(R_2)  +2\alpha  \int_{R_2}^{R_3}  W(r)e^{-2\alpha (r-R_2)} dr  
 \end{equation}

As the entropy of $W$  is zero, we deduce  that the first term of the right hand side of \eqref{ine11} tends to  zero as $R_3\longrightarrow \infty.$ 
On the other hand, as 
$We^{-\alpha r}$ is bounded, the third term of the right hand side of \eqref{ine11} converges for $R_3\longrightarrow \infty.$
 Hence   our claim is proved and the left  hand term  of inequality \eqref{ine8}  is  bounded independently of $R_3.$ 
We let $R_3\longrightarrow\infty$ in  \eqref{ine8}. 
As the first term is independent of $R_2,$ $R_3,$  there exists a constant $C_2$  (independent of $R_2$) such that 

\begin{equation}
\label{ine12}
\int_{{B_{\sigma}(R_2)}\setminus {B_{\sigma}(R_0)}}w\leq C_2
\end{equation}

By letting $R_2$ go to infinity on \eqref{ine12} we obtain the result.
 \end{proof}  
\begin{remark}
We note that Theorem \ref{caccio-entropy} can be viewed as a generalization to manifolds with density of the result of Brooks \cite{Br}, where  the volume entropy is zero.

\end{remark}

\section{Applications of Caccioppoli's inequalities}
\label{caccio-appli}

In this section,   we will apply Caccioppoli's inequalities of Section \ref{caccioppoli}   to obtain 
some results   about entropies. \\
 When $M$ is a minimal hypersurface immersed in a manifold with constant   curvature,    R. Schoen, L. Simon, S.T. Yau  obtained  the following inequality (see the last inequality of the proof of Theorem 2 in \cite{SSY} )
 
 \begin{equation}
 \label{SSY-ineq}
 \int_{{B_{\sigma}(tR)}}|A|^p\leq\frac{\beta}{(1-t)^pR^p}|B_{{\sigma}}(R)|
 \end{equation}
 
 where $\beta$ is a positive constant, $t\in(0,1),$ $p\in\left(0,4+\sqrt{\frac{8}{n}}\right).$ This  inequality  
 implies a comparison between the entropy  of ${\mathcal T}_p$ and the volume entropy of   $M.$ 
 In the same spirit, using Caccioppoli's inequality \eqref{i3} in  Theorem \ref{caccio-theo}, we deduce a similar comparison
 result.\\
  In this section we assume that the ambient manifold $\mathcal{N}$ 
is an orientable Riemannian manifold with bounded sectional curvature.\\

   \begin{theorem}
 \label{comparison-stable}
 Let $M$ be a  complete, noncompact  stable  hypersurface  with constant mean curvature,  immersed  in  $\mathcal N.$ Assume $x\in[1,1+\sqrt{\frac{2}{n}}),$ then one has:

 $$\mu_{{\mathcal T}_{2x+2}}\leq\mu_{M}.$$

\end{theorem}

\begin{proof}  
 Fix  $t\in(0,1)$ and  choose a radial test function $f\equiv1$   on  $B_{{\sigma}} (tR)$, $f\equiv 0$ on $M\setminus B_{{\sigma}}(R)$ and linear on the annulus $B_{{\sigma}}(R)\setminus B_{{\sigma}}(tR).$ 
By a straightforward computation,  inequality \eqref{i3} with $K\subset B_{{\sigma}}(tR)$ applied to $f$ yields

\begin{equation}
\label{caccio-ball}
\beta_1\int_{B_{{\sigma}}(t R)} \varphi^{2x+2}\leq |B_{{\sigma}}(R)|\left(\frac{\beta_2}{(1-t)^{2x+2}R^{2x+2}}+\beta_3\right).
\end{equation}

In \eqref{caccio-ball}, we take  the logarithm,  divide by $R$ and take the limit  of both sides for $R\longrightarrow\infty$ and we   obtain the result.

\end{proof}

\begin{remark} 
\nero The H\"{o}lder inequality gives for any $p>0$ and $q\ge p$: 
\begin{equation}\label{holder-entrop}
\frac{1}{p}\mu_{{\mathcal T}_{p}}\leq (\frac{1}{p}-\frac{1}{q})\mu_{M}+\frac{1}{q}\mu_{{\mathcal T}_{q}}
\end{equation}
\nero Note that  inequality  \eqref{n-red-exp} yields the following comparison between entropies of total curvature
$$\mu_{{\mathcal T}_{2x+2}}\leq \mu_{{\mathcal T}_{2x}}.$$
This last inequality combined with inequality (\ref{holder-entrop}) gives an alternative proof of the inequality of Theorem  \ref{comparison-stable} 
\end{remark}

Another   application of Caccioppoli's inequality  yields:

\begin{theorem}
\label{tysk}
Let $M$ be a complete, noncompact, hypersurface immersed with constant mean curvature $H\not = 0$  in a manifold $\mathcal N$ with constant   curvature $c.$ Assume $M$ has finite index and dimension $n\leq 5.$ Then  provided either

(1) $c=0$, or $c=1$,    $x\in[1,x_2)$ 

or 

(2) $c=-1,$ $\varepsilon>0,$ $x\in[1,x_2-\varepsilon],$ $H^2\geq g_n(x).$

one has 

\begin{equation}
\label{equivalence}
\int_M\varphi^{2x}<\infty \ {\rm if \ and \ only \ if} \  \mu_{{\mathcal T}_{2x}}\equiv0.
\end{equation}

\end{theorem}

\begin{proof} 
We first observe that $x_2 $,   given in \eqref{roots}, is greater than one   if and only if $n\leq 5$.\\
The first implication is clear by the definition of entropy. Vice-versa, as inequality \eqref{j-bis} holds in $M,$  we can apply (2) of Theorem \ref{caccio-entropy} with $w=\varphi^{2x}$  in order to have the result. 
 \end{proof}
 A direct consequence of   Theorem  \ref{tysk} is the following    result  in the spirit of  a result of  J.  Tysk concerning 
 Euclidean minimal hypersurfaces (see the main theorem  of \cite{Ty}).
 \begin{corollary} \label{cortysk}
Let $M$ be a complete, noncompact, hypersurface immersed with constant mean curvature $H\not = 0$  in a manifold $\mathcal N$ with constant   curvature $c.$ 
Assume $M$ has finite index and dimension $n\leq 5.$   Then  provided either\\
 (1) $c=0$ or  $c=1$
  
  or 
  
 (2)  $c=-1$ and $H^2 \geq  g_n(1)$ ($g_n(x)$ is defined in \eqref{function-g}) \\
 one has 
$$\mu_{{\mathcal T}_{2}}=0 \ {\rm implies}\  \int_M\varphi^{n}<\infty$$ 
\end{corollary}
\begin{proof} 
 By Theorem \ref{tysk}, we have  $\int_M\varphi^{2}<\infty $. Using the reverse Holder inequality of Theorem  \ref{reduction-exponent} for $x=1$, we obtain $\int_M\varphi^{4}<\infty.$ Using Holder inequality, we obtain $\int_M\varphi^{3}<\infty.$
 Using the reverse Holder inequality of Theorem  \ref{reduction-exponent} for $x=\frac{3}{2}$, we obtain $\int_M\varphi^{5}<\infty.$
\end{proof}
\begin{remark} 
\nero As $\mu_M\geq \mu_{{\mathcal T}_{2x}},$ one can replace  the hypothesis $\mu_{{\mathcal T}_{2x}}\equiv 0$    
 in Theorem \ref{tysk}    and   the hypothesis  $\mu_{{\mathcal T}_{2}}\equiv 0$  in Corollary \ref{cortysk} by 
$\mu_M\equiv 0.$ 

\nero  P. B\'erard, M. do Carmo and W. Santos  \cite{BDS} proved that  if $M$ is a complete  hypersurface in 
$\h^{n+1},$ with constant mean curvature $H^2<1,$ such that $\int_M\varphi^n<\infty,$ then $M$ has finite index. 
Notice that the converse is false  as it is shown  by the examples of A. da Silveira \cite{S}.
\end{remark}

\section{Some  answers to do Carmo's question }
\label{docarmo-section}

In this Section we apply our previous results in order to answer, at least in some cases, to the following do Carmo's question 
\cite{DoC1} : {\em is a noncompact, complete, stable, constant mean curvature  hypersurface of $\Rn,$ $n\geq 3,$  
necessarily  minimal? }

It will be clear in  the following in which cases we are able to give an answer.

   \begin{theorem}
\label{carmo-volume-entropy}
There is no complete, noncompact, finite index   hypersurface $M$   immersed in a manifold $\mathcal N,$
provided  the mean curvature function $H$ 
  satisfies $ nH^2+Ric(\nu,\nu)\geq \delta,$ where $\delta$ is a constant such that $\delta > \frac{\mu_M^2}{4}.$
\end{theorem}

\begin{proof}[Proof of Theorem \ref{carmo-volume-entropy}]
Assume such $M$ exists. As $M$ has finite index, there exists a compact $K$ in $M$ such that $M\setminus K$ is stable. Therefore, for any 
$f\in C_0^{\infty}(M\setminus K),$ one has

\begin{equation}
\label{carmo-one}
0\leq Q(f,f)=\int_{M\setminus K}|\nabla f|^2-(|A|^2+Ric(\nu,\nu))f^2.
\end{equation}

As $|A|^2+Ric(\nu,\nu)\geq nH^2+Ric(\nu,\nu)\geq\delta,$ \eqref{carmo-one} yields 

\begin{equation}
\label{carmo-two}
0\leq \int_{M\setminus K}|\nabla f|^2-\delta\int_{M\setminus K}f^2,
\end{equation}

that is $\lambda_0(M\setminus K)\geq \delta.$    This contradicts  inequality \eqref{brooks-ineq} of
 Brooks'  Theorem.
\end{proof}

\begin{remark} A result  weaker  than  Theorem \ref{carmo-volume-entropy} was proved by  M. do Carmo and D. Zhou \cite{DZ}. 
\end{remark}

A straighforward application of Theorem \ref{carmo-volume-entropy} gives the following result.
\begin{corollary}
\label{carmo-euclidean-hyp}
  If $M$ is  a constant mean curvature hypersurface of finite index  immersed in a space of constant curvature $c,$ 
  then   $n(H^2 + c ) \leq \frac{\mu_M^2}{4}    $.  In particular,  if $H^2 + c >0$, the volume growth of $M$ is exponential. 
 \end{corollary}

  As an   application of  Corollary  \ref{carmo-euclidean-hyp} we obtain the following result.

 \begin{corollary}
There is no complete, noncompact, finite index, constant mean curvature  hypersurface $M$  immersed in a manifold $\mathcal N,$ with 
$\mu_M\equiv0,$ provided either\\
(1) ${\mathcal N}={\mathbb S}^{n+1}$ or\\
(2) ${\mathcal N}={\mathbb R}^{n+1}$ and  $H\not=0,$ or\\
(3) ${\mathcal N}={\mathbb H}^{n+1},$ $H>1.$
\end{corollary}

  One can replace the entropy  
   $\mu_M$  by the entropy of the ambient space $\mu_\mathcal{N}$   by imposing some  additional geometric conditions  on $M$ and $\mathcal N,$  that  guarantee the embeddedness   of the tube around $M.$ 
  
 \begin{theorem}
\label{embedded-hyp-eucl} 
There is no complete, noncompact, finite index $M $ that is   properly embedded in a simply connected manifold $\mathcal N$, where 
$M$ and $\mathcal N$ have bounded curvature  and
provided   $H\geq \delta_1 > 0$ and  $ nH^2 + Ric(\nu,\nu) \geq \delta_2 >  \frac{\mu_{\mathcal N}^2}{4} $   for some
positive $\delta_1,  \delta_2$.\end{theorem}

\begin{proof}
  One applies Corollary \ref{volume-estimate2} and Theorem \ref{carmo-volume-entropy}
\end{proof}

As  an immediate consequence of Theorem \ref{embedded-hyp-eucl} one has the following result.

\begin{corollary}
 If $M$ is a  properly embedded hypersurface of bounded  curvature and constant mean curvature with finite index in a space form of curvature $c$ :\\
 \nero If  $c\geq 0,$  then $c=0$ and $M$ is minimal. \\
\nero  If $c< 0$, then   $H^2 \leq - c\left( \frac{n}{4}+1 \right)  $
  \end{corollary}

  In Theorem \ref{carmo-volume-entropy}, when $\mathcal N$ is either ${\mathbb R}^{n+1}$ or ${\mathbb H}^{n+1},$ 
one can replace the volume entropy  by some  total curvature entropy, provided some restriction on the  dimension and on the 
mean curvature of $M$ are satisfied.

\begin{theorem}
\label{general-theo}
There is no complete noncompact stable  hypersurface $M$ with constant mean curvature $H$ in  a manifold 
${\mathcal N},$ $n\leq 5,$  with  $\mu_{{\mathcal T}_{2x}}=0,$ provided either 

(1) ${\mathcal N}={\mathbb R}^{n+1},$  $x\in[1,x_2),$ $H>0,$  

or 

(2) ${\mathcal N}={\mathbb H}^{n+1},$  $\varepsilon>0,$ $x\in[1,x_2-\varepsilon],$ $H^2> g_n(x)$  ($g_n(x)$ is defined in \eqref{function-g}).

\end{theorem}

\begin{proof} First notice that,    by Theorem  \ref{tysk},  the hypothesis 
$\mu_{{\mathcal T}_{2x}}=0.$  implies      $\int_M\varphi^{2x}<\infty.$   
  One applies Corollary 6.3 of \cite{INS}  in order to obtain that $M$ is totally umbilic.   In case (1), it follows that  $M$ is contained either in a sphere or in a plane. As $M$ is complete noncompact, $M$ is contained is a plane, then $H=0.$  Contradiction. In case (2), it follows that $M$ is contained either in a sphere, or in a horosphere, or in a equidistant sphere. The inequality $H^2>g_n(x)\geq 1$ yields that $M$ can be only contained in a sphere. 
 As $M$ is complete and noncompact, this is a contradiction.
\end{proof}

\begin{remark} \nero It is worthwhile to note   that the condition $\int_M\varphi^{2x}<\infty$ in Theorem \ref{general-theo}  is equivalent to  the apparently weaker condition $\mu_{{\mathcal T}_{2x}}=0$ (see Theorem \ref{tysk})\\
\nero In  Theorem \ref{general-theo} , under the same conditions, if ${\mathcal N}$ 
has constant sectional curvature $c\leq 0$ but   not necessarily  simply connected, one prove similarly that $M$ is totally umbilical. \\
\end{remark}

 \section{Appendix}

We develop here, the missing parts of the principal arguments of the proof of  Theorem \ref{tubeplonge}.  
 In particular we will prove estimate (\ref{phi-estimate}) and
equality (\ref{H-P}). 

The main reason for this Appendix 
is the lack of references for the  computations of the mean curvature equation in general Riemannian  manifolds, 
and for the uniform   estimates  of its  coefficients. The mean curvature equation has been extensively studied in Euclidian spaces, constant curvature spaces 
and more  recently for the particular case of  constant mean curvature  surfaces in homogenous spaces (see for instance \cite{RST}). A similar mean curvature equation was obtained in Fermi coordinates for minimal surfaces 
in \cite{CM}.  We  first, give  a proof of  the cheesebox argument,  
then we compute  the mean curvature equation and deduce the desired estimates 
of its coefficients.  
This will allow us to fill the gaps in the proof of Theorem \ref{tubeplonge}.

 \subsection{The cheesebox argument}\label{cheese}
 In this paragraph we recall the cheesebox argument and show how it implies  estimate \eqref{phi-estimate}.
 We use the same notations as in  the proof of Theorem \ref{tubeplonge}. \\
 Recall that  the section  $P $ is  defined as the graph of the function $\phi$ defined at the beginning of the proof of the second step of Theorem \ref{tubeplonge}.\\
 \indent  We first study  the case where the ambient space  $\mathcal N$ is $ \mathbb{R}^{n+1}$.\\

   \begin{figure}[lh]

\mbox{  \hskip 1 in  %
 \includegraphics[scale = .25]{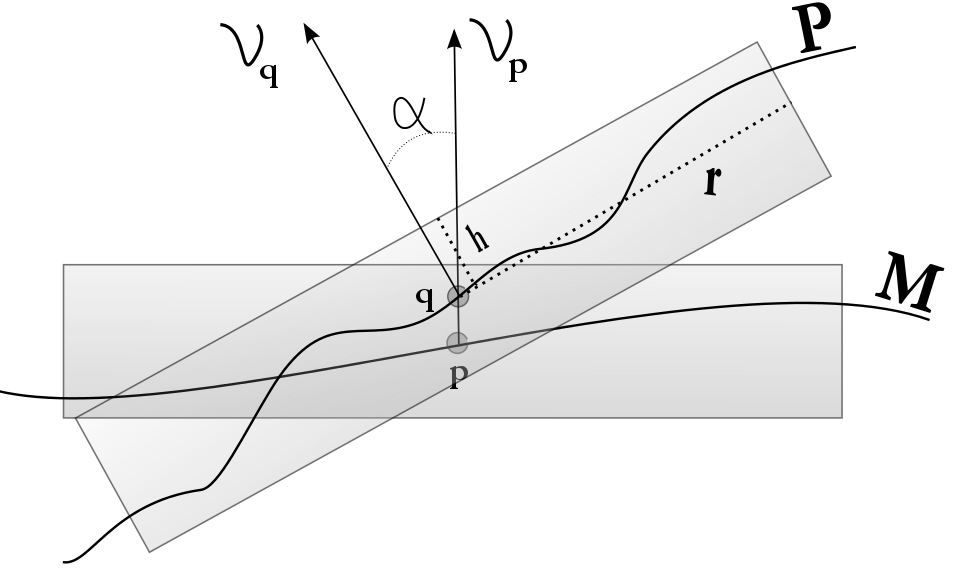}}
 \caption[]{ Intersection of the cheeseboxes  at $p\in M$ and  at $q\in P$ with $\mathbb{R}\nu_p \oplus  \mathbb{R}\nu_q$  } 
\end{figure}
 Let    $p $  be a point in $M$ and  $q $ be   a point   above $p$ 
in  the section $P$ (see Figure 3).  Consider the following  boxes of $\mathcal{N}:$

$$ C_p(\rho, h) := exp_p\left( D(\rho) \times ]0,h[\right) \ {\rm and}\  C_q(\rho, h) := exp_q\left( D(\rho) \times ]0,h[\right). $$
We choose $\rho$ as in Theorem \ref{tubeplonge}  so that $M$ and $C_p(\rho, h)$ (respectively $P$ and $C_q(\rho, h)$)
intersect only on the boundary side  $ \partial D \times [0,h]$ of  $C_p(\rho, h)$ (resp. on the boundary side  $ \partial D \times [0,h]$ of  
$C_q(\rho, h)$).
 The intersection of $C_p(\rho, h) \cup C_q(\rho, h)$ with the   plane  through $p$ generated by  the normal vectors  $\nu(p)$  
 and $\nu(q)$,    is the union of two Euclidean rectangles (see   Figure 3). \\
For a sufficiently  small $\rho,$ we may choose the height of the box to be  proportional to  the square of its radius:  $h = c\rho^2$ for some constant $c$
 depending on   the $C^1$-norm of the  the second fundamental form  $A$ of $M$.  In fact, in a neighborhood of any  $x_0\in M$, $M$  is the graph 
of the height function  $h$ defined on a   ball  $B_{x_{0}}(\rho)$ of the tangent space at $x_0$  of sufficiently small radius $\rho,$
  and such that for any $x\in B_{x_{0}}(\rho)$,   $|h(x)| \leq C|x|^2$   where $C$ is a uniform constant.

Indeed, let us clarify the relation between the second fundamental form of the aforementioned graph defined by the  height  function $h(x_1,\cdots, x_n)$  and the Hessian of $h.$ Notice that $h(0)=0$ and $\nabla h (0)=0.$

Using the Einstein's convention on indices, we have $|A|^{2}=g^{ik}g^{jl}A_{il}A_{kj}=A_{i}^{j}A_{j}^{i}$ ,
and
$$A_i^j  = \frac{h_{ik}}{W}g^{kj}, \quad {\rm with } \  W:= \sqrt{ 1+ f}, \quad  f:=|\nabla h|^2\ {\rm and}\  g^{ij}:= \delta^{ij} - \frac{h_i h_j}{W^3}$$
Therefore, 
$$A_i^j  = \frac{h_{ij}}{W}- \frac{  h_{ki}h_k h_j  }{W^3}$$
and computing $|A_i^j h_j |$, we get :
$$ |A_i^j h_j |= \left|\frac{f_i}{2W}\left( 1 -\frac{f}{W^2}    \right)\right|$$
Since $A$ is bounded, there exists a positive constant $c$ such that $|A_i^j | \leq \frac{c}{2\sqrt{n}}$, and  applying   Cauchy-Schwarz's inequality, we derive:
$$ 2|A_i^j h_j |= \left|\frac{f_i}{(1+f)^{3/2}}\right|\leq   c \,\sqrt{f} $$
Standard comparison between the solutions of the previous  differential inequality and the corresponding differential equality, together with the initial condition $ \nabla h(0)=0$, shows  that  the condition 
 $ \sum |x_i|^2  \leq \rho $ implies   $|\nabla h | \leq \frac{c\, \rho  }{\sqrt{4-(c\,\rho)^2}}$.\\
 Thus, $h$ is $C^1$-uniformly bounded on the ball $B_{x_{0}}(\rho)$ and $W$ is  also uniformly bounded.
 Finally,  $h_{ij} \leq C$ on $B_{x_{0}}(\rho)$ (where $C$ is a constant depending on $c$ and $\rho$). In conclusion, since $h(0)=0$ and $\nabla h(0) =0$, we have  $|h(x)|\leq C|x|^2 $ on $B_{x_{0}}(\rho)$.
\\
  
  We  may  also suppose 
  that  the function $\phi$ defining the section $P$ satisfies $\phi   \leq \frac{h}{2}$
 (in Figure  2, $\phi(p):= d_{\mathbb{R}^{n+1}}(p,q)$).
Let $\alpha$   be the angle  defined by $\tan \alpha = |\nabla \phi|(p)$
 (Figure 2 represents a limit case for which $M$ and $P$ necessarily intersects for  any $q'$ such that $d_{\mathbb{R}^{n+1}}(p,q' )\leq d_{\mathbb{R}^{n+1}}(p,q)$ or  any 
 $\nu(q')$ such that $\left<\nu(p),\nu(q')\right> \leq \cos \alpha$).\\
 
 For a given $\rho$, if $d_{\mathbb{R}^{n+1}}(p,q)$ is small enough then $|\nabla \phi |\leq 1$
 unless $M$ and $P$ intersects. Thus  $\alpha \leq \alpha_0= \frac{\pi}{4}.$ 
Then   elementary  plane geometry gives

   $$\rho\sin\alpha  \leq \rho \sin\alpha_0  \leq h \left(1 + \cos\alpha_0     \right)  + \phi(p) \cos\alpha_0  $$
    thus 
     $$\rho\sin\alpha \leq \left(h + \phi\right)\cos\alpha_0 + h \leq 3h$$
  hence
  $$\frac{|\nabla\phi|(p)}{\sqrt{1+|\nabla\phi|^2(p)}}\leq 3c\rho.$$
which yields $ |\nabla \phi|(p) \leq 6c\rho$.\\
Since, by hypothesis,    $\phi(p) \leq \frac{c}{2}\rho^2$ and since  the  inequality holds for any point $p\in \Omega \subset M$, we obtain 
 \begin{equation}\label{phi}
  \|\phi\|_1 := \left(\sup_{p\in \Omega  } 
  | \phi|  + \sup_{p\in \Omega  } 
  |\nabla\phi| \right)   \leq O(\rho).
 \end{equation} 
  \indent Consider the  general case, where the    ambient space  $\mathcal N$ is not necessarily Euclidean. 
 Since $\mathcal N$ has  bounded curvature,  there exists  a radius $\rho_0$ 
 depending on the  curvature of $\mathcal N,$  such that for  each point $p\in \mathcal N$, there is a harmonic coordinate chart 
 $\psi_p^{-1},$    such that 
 $\psi_p: U \left(:=  B_0^{\mathbb{R}^{n+1}\left( \rho_0\right)}\right) \subset  \mathbb{R}^{n+1} \longrightarrow  
 V := \psi_p\left(U\right)\subset {\mathcal N} $,  and 
 the pulled-back metric $g_{\mathcal N} $ is $C^{1,\alpha}$- regular,   
 $C^{1,\alpha}$- close to the Euclidean one.
 The  diffeomorphism $\psi_p$  is $C^1$- uniformly bounded in these coordinates.
   The previous result   concerning  Euclidean cheeseboxes  applies to 
  $\psi_p^{-1}(M)\cap V$ and $\psi_p^{-1}(P)\cap V$   to prove   the  $C^1$-uniformly boundedness of $\phi\circ \psi_p$  with respect to  $p\in M$.  Finally since   $\psi_p$ is  $C^1$-uniformly bounded with respect to $p$, so is    $\phi$ .
 In conclusion  there exists  a radius $\rho_0,$ depending on the  curvature of $\mathcal N$ and $M,$   such that for  each point 
 $p\in M,$ there  a cheesebox    of $M$ around $p$   of radius $\rho_0$ and height $c\rho_0^2$ in harmonic coordinate charts. 
For details about the theory of harmonic coordinates  see for instance the  survey \cite{HH} and the references therein.

\subsection{Cmc equation of a section of the normal bundle of $M$}\label{cmcequation}
 Notations are the same as in  the previous paragraph and  Section \ref{tubes}.
 Our purpose is to  compute the mean curvature  $H_P$ of the section $P$ in a neighborhood 
 of $q\in P$, in terms of   local coordinates around $p\in M$.   More precisely  we will show how to obtain the expansion  \eqref{H-P}    of $H_P$   in the proof of Theorem \ref{tubeplonge}.\\
Let $\psi_p$ be a parametrization of a neighborhood $V_p$ 
 of $p$ in $\mathcal{N}$  as given in previous  paragraph:
 $ (\psi_p :B^{{\mathbb R}^{n+1}}_0(R)\subset \mathbb{R}^{n}\times \mathbb{R} \longrightarrow V_p \subset  \mathcal{N})$
 with   
 $\psi(0) = p$ and $\psi (  B^{{\mathbb R}^{n+1}}_0(R) \cap  \mathbb{R}^{n}\times \{0\} )  = M\cap V_p$. 
  The  local section $P\cap V_p,$ being in a cheesebox,   is parametrized by a graph of a function   $\phi :
    \mathbb{R}^{n}\times \{0\} \longrightarrow \mathbb{R}$.   Indeed  $\psi_p^{-1}(P\cap V_p)$ and 
  $\psi_p^{-1}(M\cap V_p)$   are $C^{1}$-close in the pulled-back metric $g_{\mathcal{N}}$.  For simplicity 
we identify the metric  $g_\mathcal{N}$  of $\mathcal{N}$ with its pulled-back  $\psi^*(g_\mathcal{N})$ on the Euclidean ball $B^{{\mathbb R}^{n+1}}_0(R)$.  We denote by  $\{e_\alpha\}_{\alpha = 1,\cdots, n+1}$ the standard basis  of  ${\mathbb R}^{n+1}$ and 
by $\{e_i\}_{i = 1,\cdots, n}$ the standard  basis  of ${\mathbb R}^{n}\times\{0\}$.
With some abuse of notations, we identify $ (x,\phi(x)) $ with its image $\psi_p(x,\phi(x)) $  and 
derivatives with respect to $e_i$ of  a function $f$  will be denoted by $f_{,i}$. 

 \subsubsection{Mean curvature equation of the  section $P$}
 We first   choose an adapted   frame tangent to P such that the first $n$ vectors
 $\{ f_i\left( x  ,\phi\left(x\right)\right) := e_i + \phi_{,i}\left(x,\phi\left(x\right)\right) e_{n+1}\}_{i=1,\cdots n} $
 are tangent to $P $ at $\left( x,\phi\left( x\right)\right)$
 and the last vector is the unit normal field   $\nu_P(x,\phi(x))$  to   the graph of  $\phi.$ 
  In fact $\nu_{P}$ is a unit vector,  solution  of the system of    linear equations  given by: 
  $\Big\{g_{\mathcal{N}}\left(\nu_P,  F_{i}\right) = 0\Big\}_{ i = 1, \cdots, n}$
    and an easy computation yields :
 \begin{equation}\label{len}
  \nu^\alpha := \frac{1}{W}\left( - g^{\alpha i}\phi_{,i} + g^{\alpha\  n+1}\right).
  \end{equation}  
 where $g^{\alpha\beta}$ is the  inverse matrix of
  $g_{\alpha\beta} := g_{\mathcal{N}}\left(e_{\alpha},e_{\beta}\right), \alpha,\beta = 1,\cdots n+1$,
 and  $W^2:=  g^{kl}\phi_{,k}\phi_{,l}-2g^{n+1 k }\phi_{,k}+g^{n+1 n+1}$ (we use in all subsequent formulas the Einstein 
  summation convention).\\
We denote by $\tilde g_{ij} := g_{\mathcal N}(f_i,f_j) $ the coefficients of the induced metric on
  $T_qP.$ 
 Therefore, replacing  each  $f_i$ by  its expression in terms of $e_i$, we obtain 
  $$ \gc_{ij} = g_{ij} + g_{n+1 (j} \phi_{,i)} + g_{n+1 n+1} \phi_{,i}\phi_{,j}$$
(where $g_{n+1 (j} \phi_{,i)}  :=  g_{n+1 j} \phi_{,i} + g_{n+1 i} \phi_{,j} $).\\

Let us now  compute    the mean curvature  equation  for $P$. 
We have:   
 \begin{equation}\label{cmc2}
   nH_P(q)  =  - div(\nu_P)(q) = -  \gc^{ij}g_{\mathcal N}(\nabla_{f_i}\nu_P, f_j )(q) 
  =    \gc^{ij}g_{\mathcal N}(\nu_P, \nabla_{f_i} f_j )(q)  \end{equation}  
 Therefore,   for $\alpha,\beta,\gamma= 1,\cdots ,n+1$ and $i,j =  1,\cdots, n$ we obtain 
    $$(\nabla_{f_i}{f_j})^\alpha =  \Gamma^\alpha_{ij} +  \Gamma^\alpha_{n+1 (j} \phi_{,i)} +
 \Gamma^\alpha_{n+1 n+1} \phi_{,i}\phi_{,j} +\phi_{,ij}\delta^{\alpha n+1}$$
We then compute $g_{\mathcal N}(\nu_P, \nabla_{f_i} f_j )$ and plug into equation \eqref{cmc2}. We obtain
\begin{equation}\label{cmc3}
 nH_P W = \gc^{ij} \left(    \phi_{,ij} +     \left(\Gamma^k_{n+1 n+1} \phi_{,k} +
  \Gamma^{n+1}_{n+1 n+1}     \right)  \phi_{,i}\phi_{,j} + \Gamma^k_{n+1 (i}\phi_{,j)}\phi_{,k}
   + \Gamma^{n+1}_{n+1 (i} \phi_{,j)}  + \Gamma^k_{ij}\phi_{,k}  + \Gamma^{n+1}_{ij}\right)
  \end{equation}
  
We  use now  harmonic charts as in the end of   last paragraph.
 In these charts,  the induced metric $g_{\mathcal{N}}$
   of $\mathcal{N}$ on  $B^{{\mathbb R}^{n+1}}_p(\rho) $ is $C^{1,\alpha}$- regular  and  $C^{1,\alpha}$ uniformly close to  the Euclidean metric.   Since  $\phi$ is  also  $C^1$ uniformely bounded (see \eqref{phi}),  the coefficients of  equation \eqref{cmc3} are $C^{0,\alpha}$ uniformly  bounded. Using  Schauder estimates, we obtain uniform  $C^\infty$ bounds on $\phi$. \\
  Notice first that, when  $\phi =0$  the same equation (\ref{cmc3}) gives the mean curvature of the zero section:
 \begin{equation}\label{cmc3bis}
 nH_M  g^{n+1 n+1}   =  g^{ij}     \Gamma^{n+1}_{ij} 
  \end{equation}
Replacing  equation \eqref{cmc3bis} in  equation \eqref{cmc3} we obtain the estimate
      \begin{equation}\label{cmc30}
 nH_P  = nH_M    +     \gc^{ij}   \phi_{,ij}  + O(\rho^\alpha) =  nH_M    +  \Delta_M\phi + O(\rho^\alpha)  \end{equation} 
which gives estimate estimate \eqref{H-P} in the proof of  theorem \ref{tubeplonge}


\textsc{Said Ilias}

{\em Universit\'e F. Rabelais, D\'ep. de Math\'ematiques, Tours

ilias@univ-tours.fr}

\textsc{Barbara Nelli}

{\em DISIM, Universit\'a dell'Aquila

nelli@univaq.it}

\textsc{Marc Soret}

{\em Universit\'e F. Rabelais, D\'ep. de Math\'ematiques, Tours

 marc.soret@lmpt.univ-tours.fr}

  \end{document}